\documentclass[11pt,a4paper,reqno]{amsart}
\usepackage{amsaddr}
\usepackage{amsmath,amsfonts,amssymb,amsthm,mathrsfs}
\usepackage[utf8]{inputenc}
\usepackage[T1]{fontenc}
\usepackage{lmodern}
\usepackage{latexsym}
\usepackage{cancel}
\usepackage{microtype}
\usepackage{caption}
\usepackage{MnSymbol}
\usepackage{xcolor}
\usepackage{enumerate}
\usepackage[normalem]{ulem}
\usepackage{bbm}
\usepackage{geometry}
\usepackage{marginnote}

\usepackage[colorlinks=true,pdfpagemode=none,urlcolor=blue,
linkcolor=blue,citecolor=blue]{hyperref}

\usepackage{graphics,tikz} 
\usepackage{caption}
\def\XXint#1#2#3{{\setbox0=\hbox{$#1{#2#3}{\int}$ }
\vcenter{\hbox{$#2#3$ }}\kern-.6\wd0}}

\newcommand{\esssup}{\mathop{\mathrm{ess\,sup}}}

\newtheorem{theorem}{\bf Theorem}[section]
\newtheorem{proposition}[theorem]{\bf Proposition}
\newtheorem{corollary}[theorem]{\bf Corollary}

\theoremstyle{definition}
\newtheorem{definition}[theorem]{Definition}

\newtheorem{remark}[theorem]{Remark}

\numberwithin{equation}{section}

\begin{document}

\title[A priori estimates  for  BSDEs with integrable data]{A priori estimates for  multidimensional BSDEs with integrable data}

\maketitle
\begin{center}
 \normalsize
 % \authorfootnotes
  TOMASZ KLIMSIAK\footnote{e-mail: {\tt tomas@mat.umk.pl}}\textsuperscript{1,2} \,\,\,
  MAURYCY RZYMOWSKI\footnote{e-mail: {\tt maurycyrzymowski@mat.umk.pl}}\textsuperscript{2}   \par \bigskip
  \textsuperscript{\tiny 1} {\tiny Institute of Mathematics, Polish Academy of Sciences,\\
 \'{S}niadeckich 8,   00-656 Warsaw, Poland} \par \medskip
 
  \textsuperscript{\tiny 2} {\tiny Faculty of
Mathematics and Computer Science, Nicolaus Copernicus University,\\
Chopina 12/18, 87-100 Toru\'n, Poland }\par
\end{center}

\begin{abstract}
We study Backward Stochastic Differential Equations on 
a probability space equipped with a  Brownian filtration.  
We assume that the terminal value and the generator at zero are merely integrable. 
Moreover, the generator is assumed to be non-increasing with respect to the value variable
(with no restrictions on the growth) and
Lipschitz continuous, with sublinear growth, with respect to the control variable.
We provide a priori estimate and stability result  for solutions to the aforementioned BSDEs.
\end{abstract}
\maketitle
%\footnotetext{{\em Mathematics Subject Classification:}
%Primary  35J10, 60J45; Secondary 35B25, 35J08, 31C25, 47G20.}

%\footnotetext{{\em Keywords:} Schr\"odinger operator}

\section{Introduction}

Let us fix   a complete probability space $(\Omega,\mathcal F,P)$  and a  $d$-dimensional Brownian motion  $B$ on $(\Omega,\mathcal F,P)$.
Let $\mathbb F$ be the standard augmentation of the filtration generated by $B$. 
We say that a pair $(Y,Z)$, consisting of $\mathbb F$-progressively measurable  processes,  is a solution to a Backward Stochastic Differential 
Equation with data $(\tau,\xi,f)$ (notation: BSDE$^\tau(\xi,f)$), where 
$\tau$ is a bounded $\mathbb F$-stopping time (terminal time),   $\xi$ is an  $\mathcal F_\tau$-measurable random vector (terminal condition),
and  $f:\Omega\times [0,\tau]\times \mathbb R^k\times \mathbb R^{d\times k}\to \mathbb R^k$ is  an $\mathbb F$-progressively measurable process
with respect to the first two variables (generator), if 
\begin{equation}
\label{eqi.1}
Y_t=\xi+\int_t^\tau f(r,Y_r,Z_r)\,dr-\int_t^\tau Z_r\,dB_r,\quad t\in [0,\tau].
\end{equation}
In  \cite{bdh} Briand, Delyon, Hu, Pardoux and Stoica  have proven that under the following weak assumptions on the data 
\begin{enumerate}
\item[(H1)] there is $\lambda\ge0$ such that
$|f(t,y,z)-f(t,y,z')|\le\lambda|z-z'|$ for $t\in[0,\tau]$, $y\in\mathbb{R}^k$, $z,z'\in\mathbb{R}^{d\times k}$,
\item[(H2)] there is $\mu\in\mathbb{R}$ such that
$\langle y-y',f(t,y,z)-f(t,y',z)\rangle\leq\mu|y-y'|^2$ for $t\in[0,\tau]$, $y,y'\in\mathbb{R}^k$, $z\in\mathbb{R}^{d\times k}$,
\item[(H3)] for every $(t,z)\in[0,\tau]\times\mathbb{R}^{d\times k}$ the mapping $\mathbb{R}^k\ni y\rightarrow f(t,y,z)$ is continuous,
\item[(H4)] either $k\ge 2$ and $\mathbb E\Big(\int^\tau_0 \sup_{|y|\le M}|f(r,y,0)|\,dr\Big)<\infty$ for any $M>0$
or $k=1$ and $\int_0^\tau|f(r,y,0)|\,dr<\infty$ for any $y\in\mathbb R$,
\item[(H5)] $\xi\in L^p(\mathcal{F}_\tau)$, $f(\cdot,0,0)\in L^{p}_{\mathbb F}(0,\tau)$,
\item[(Z)] there exists an $\mathbb{F}$-progressively measurable process $g\in L^1_{\mathbb F}(0,\tau)$ and $\gamma\ge 0$, $\kappa\in [0,1)$ such that
\begin{align*}
|f(t,y,z)-f(t,y,0)|\le\gamma(g_t+|y|+|z|)^\kappa,\quad t\in [0,\tau], \,y\in\mathbb{R}^k,\,z\in\mathbb{R}^{d\times k},
\end{align*}
\end{enumerate}
with $p=1$ in (H5), there exists a solution $(Y,Z)$ to \eqref{eqi.1} such that $Y$ is of class (D) and  $Z\in \mathcal H^s_{\mathbb F}(0,\tau),\, s\in (0,1)$.
Moreover, under these assumptions there exists at most one solution $(Y,Z)$ to \eqref{eqi.1} such that $Y$ is of class (D). 
It is  interesting   that, although almost  20 years have passed since the publication of the above theorem, 
surprisingly such  fundamental result  as  a priori estimate for solutions to \eqref{eqi.1} is still missing from the literature 
(besides the special case when $f$ is independent of $z$-variable). 
The aim of the  present paper is to  fill this gap. 

We shall prove (see Theorems \ref{th.1},\ref{th.2}) that for any
$a,b\in (0,1)$  there  exists  a continuous strictly increasing function  $\varphi:\mathbb R^+\to\mathbb R^+$, with $\varphi(0)=0$,
which depends only on $\esssup\tau,\lambda,\mu,\gamma,\kappa,a,b, \|g\|_{L^1_{\mathbb F}}$,
such that 
\begin{equation}
\label{eqi.2}
\begin{split}
\sup_{\sigma\le\tau}\mathbb E|Y_\sigma-\bar Y_\sigma|+\mathbb E\sup_{t\le \tau}|Y_t-\bar Y_t|^a&+\mathbb E\Big(\int_0^\tau|Z_r-\bar Z_r|^2\,dr\Big)^{b/2}\\&\quad
\le  \varphi\big(\mathbb E|\xi-\bar\xi|+\mathbb E\int_0^\tau|f-\bar f|(r,\bar Y_r,\bar Z_r)\,dr\big)
\end{split}
\end{equation}
for any  solutions $(Y,Z), (\bar Y,\bar Z)$ of BSDE$^\tau(\xi, f)$, BSDE$^\tau(\bar\xi,\bar f)$, respectively,
with $Y,\bar Y$ being of class (D) (here $\bar\xi,\bar f$ satisfy the  same assumptions as $\xi,f$).
As a corollary to the above a priori estimate, we get stability result for BSDEs with $L^1$-data.

As far as we know the only paper  concerned with stability results 
for BSDEs with $L^1$-data (we omit in this comment papers with generators independent of $z$-variable) 
is  the paper  by S.J. Fan \cite{Fan}, where the author has proven the following convergence
\begin{equation}
\label{eqi.2fan}
\begin{split}
\sup_{t\le T}\mathbb E|Y_t-Y^n_t|+\mathbb E\Big(\int_0^\tau|Z_r- Z^n_r|^2\,dr\Big)^{b/2}\to 0,
\end{split}
\end{equation}
provided $\mathbb E|\xi-\xi_n|\to 0$ and $|f-f_n|\le \varepsilon_n\searrow 0$.
Here $(Y^n,Z^n)$ is a solution to BSDE$^T(\xi_n,f_n)$, with $(\xi_n,f_n)$ satisfying the same assumptions
as $(\xi,f)$, and $(\varepsilon_n)$ is a decreasing sequence of positive numbers. We see that 
\eqref{eqi.2} readily implies \eqref{eqi.2fan}, and even stronger convergence, without assuming 
boundedness of $|f-f_n|$ (note that S.J. Fan considered even weaker than  (H2) one-sided Osgood condition).   
In \cite{Fan} the author conjectured that in general a priori estimates for BSDEs with $L^1$-data cannot hold
(see the comments in the first paragraph on page 1863 in \cite{Fan}). Our main result  disproves this  conjecture.

From the theoretical and practical  point of view  a priori estimates and stability results describe one of the most
fundamental features of any type of equations.  Here, we would like to mention just about one crucial application
of our results. The result  by Briand, Delyon, Hu, Pardoux and Stoica    allow one, among others, to define, for any fixed  bounded stopping times $\alpha\le\beta$, the following operator 
(so called nonlinear expectation)
\[
\mathbb E^{f}_{\alpha,\beta}:L^1(\mathcal F_\beta)\to L^1(\mathcal F_\alpha),\quad \mathbb E^f_{\alpha,\beta}(\eta):=Y^{\beta,\eta,f}_\alpha,
\] 
where $Y^{\beta,\eta,f}$  is a process of class (D) being  the first component of a solution to BSDE$^\beta(\eta,f)$, with $k=1$.
The concept of {\em nonlinear expectation} has been introduced by Peng \cite{Peng}, originally for $p=2$ and Lipschitz continuous $f$,
and appeared to be a crucial notion in models of mathematical finance and control theory.   
Thanks to the results in \cite{bdh} this notion is also well defined on  $L^1$.
However, to apply this operator  in practice, some basic properties of it   are indispensable.  
One of them is  stability, which is still missing in the literature,  i.e. we ask what one can say about the difference
\[
\mathbb E|\mathbb E^{f}_{\alpha,\beta}(\eta_1)-\mathbb E^{f}_{\alpha,\beta}(\eta_2)|
\]
for $\eta_1,\eta_2\in L^1(\mathcal F_\beta)$.  This leads us to stability results for BSDEs with $L^1$-data.
%which is also a missing subject in the literature. 
%In the following note we shall  fill this gap.
As a corollary to our main results we have the following inequality
\begin{equation}
\label{eqi.j}
\mathbb E|\mathbb E^{f}_{\alpha,\beta}(\eta_1)-\mathbb E^{f}_{\alpha,\beta}(\eta_2)|\le \varphi(\mathbb E|\eta_1-\eta_2|),
\end{equation}
for any $\eta_1,\eta_2\in L^1(\mathcal F_\beta)$, where $\varphi$ is described in \eqref{eqi.2}.

\section{Backward SDEs with $L^p$-data: the case $p>1$}

For any $n\ge 1$ and  $x\in\mathbb{R}^n$ by $|x|$ we denote the euclidean norm of the vector $x$. 
%As mentioned in the introduction, $\mathcal{T}$ stands for the set of all stopping times taking values in $[0,T]$. What is more, for $\alpha,\beta\in\mathcal{T}$, $\mathcal{T}_{\alpha,\beta}:=\{\tau\in\mathcal{T},\,\alpha\le\tau\le\beta\}$, $\mathcal{T}_{\alpha}:=\mathcal{T}_{\alpha,T}$, $\mathcal{T}^{\beta}:=\mathcal{T}_{0,\beta}$.
Let $\beta$ be a bounded stopping time and $p> 0$.
By $\mathcal{S}^p_{\mathbb{F}}(0,\beta)$ we denote the set of all $\mathbb{F}$-progressively measurable $\mathbb{R}^k$-valued processes $Y$
such that $\mathbb E\sup_{0\le t\le\beta}|Y_t|^p<\infty$. We set
\[
||Y||_{\mathcal{S}^p_{\mathbb{F}}(0,\beta)}:=\big(\mathbb E\sup_{0\le t\le\beta}|Y_t|^p\big)^{\frac{1}{p}},\, p>1,\quad 
|Y|_{\mathcal{S}^p_{\mathbb{F}}(0,\beta)}:=\mathbb E\sup_{0\le t\le\beta}|Y_t|^p,\, p\in (0,1).
\]
Let $r,q\ge 1$. By $L^{r,q}_{\mathbb{F}}(0,\beta)$ we denote the set of all $\mathbb{F}$-progressively measurable, 
$\mathbb{R}^k$-valued processes $X$ such that
\[
||X||_{L^{r,q}_{\mathbb{F}}(0,\beta)}:=\Bigg(\mathbb{E}\Big(\int^{\beta}_{0}|X_r|^r\,dr\Big)^{\frac{q}{r}}\Bigg)^{\frac{1}{r}}<\infty.
\]
$L^r_{\mathbb{F}}(0,\beta)$ is the shorthand for $L^{r,r}_{\mathbb{F}}(0,\beta)$.

Let $\mathcal{G}\subset\mathcal{F}$ be a $\sigma$-field. $L^r(\mathcal{G})$ 
denotes the set of all $\mathcal{G}$-measurable random vectors $X$ such that $\mathbb{E}|X|^r<\infty$.
By $\mathcal{H}_{\mathbb{F}}(0,\beta)$ we denote the space of all $\mathbb{F}$-progressively measurable 
$\mathbb{R}^{d\times k}$-valued processes $Z$ such that $P(\int^{\beta}_{0}|Z_r|^2\,dr<\infty)=1$.
$\mathcal{H}^s_{\mathbb{F}}(0,\beta)$, $s>0$, is a subspace of  $\mathcal{H}_{\mathbb{F}}(0,\beta)$
consisting of  $Z$ satisfying $\mathbb{E}\Big(\int^{\beta}_{0}|Z_r|^2\,dr\Big)^{\frac{s}{2}}<\infty.$
We set
\[
||Z||_{\mathcal{H}^s_{\mathbb{F}}(0,\beta)}:=\Bigg(\mathbb{E}\Big(\int^{\beta}_{0}|Z_r|^2\,dr\Big)^{\frac{s}{2}}\Bigg)^{\frac{1}{s}},\, s>1,\quad
|Z|_{\mathcal{H}^s_{\mathbb{F}}(\alpha,\beta)}:=\mathbb{E}\Big(\int^{\beta}_{0}|Z_r|^2\,dr\Big)^{\frac{s}{2}},\, s\in (0,1).
\]
We say that $\mathbb{F}$-progessively measurable process $X$ is of class (D) on $[[0,\beta]]$ 
if the family $\{|X_{\tau}|,\,\tau \text{ is a stopping time, } \tau\le\beta\}$ is uniformly integrable. 
For $p\ge 1$, by $\mathcal{D}^p_{\mathbb{F}}(0,\beta)$ we denote the set of all $\mathbb{F}$-progressively measurable, 
$\mathbb{R}^k$-valued processes $Y$ such that $|Y|^p$ is of class (D) on $[[0,\beta]]$. 
We equip $\mathcal D^p_{\mathbb{F}}(0,\beta)$ with the norm
\[
||Y||_{\mathcal{D}^p(0,\beta)}:=\Big(\sup_{\sigma\le\beta}\mathbb E|Y_{\sigma}|^p\Big)^{\frac{1}{p}},
\]
where the supremum is taken over the set of stopping times $\sigma$.
Throughout the paper, we adopt the convention that any $l$-dimensional random vector $X$, with $l<k$,
is considered as a member of  the  space of $k$-dimensional   random vectors by the inclusion operator 
\[
X=(X^1,\dots X^l)\mapsto (X^1,\dots,X^l,0,\dots,0)\in\mathbb R^k.
\]

Let $\hat{\xi}$ be a $k$-dimensional $\mathcal{F}_{\beta}$-adapted random vector. 

\begin{definition}
We say that a pair $(Y,Z)$ of $\mathbb{F}$-adapted processes is a solution to  backward stochastic differential equation
on the interval $[[0,\beta]]$ with right-hand side $f$ and terminal value $\hat{\xi}$ (BSDE$^{\beta}(\hat{\xi},f)$ for short) if
\begin{enumerate}
\item[(a)] $Y$ is a continuous process  and $Z\in\mathcal{H}_{\mathbb{F}}(0,\beta)$,
\item[(b)] $f(\cdot,Y,Z)$ is $\mathbb F$-progressively measurable and $\int^{\beta}_{0}|f(r,Y_r,Z_r)|\,dr<\infty$,
\item[(c)] $Y_t=\hat{\xi}+\int^{\beta}_t f(r,Y_r,Z_r)\,dr-\int^{\beta}_t Z_r\,dB_r$, $t\in[0,\beta]$.
\end{enumerate}
\end{definition}

%Let $V\in\mathcal{V}_{0,\mathbb F}(\kappa,\beta)$. 

%\begin{definition}
%We say that a pair $(Y,Z)$ of $\mathbb{F}$-adapted processes is a solution to  backward stochastic differential equation
%on the interval $[[\kappa,\beta]]$ with right-hand side $f+dV$ and terminal value $\hat{\xi}$ (BSDE$^{\kappa,\beta}(\hat{\xi},f+dV)$ for short) if
%$(Y-V,Z)$ is a solution to BSDE$^{\kappa,\beta}(\hat{\xi},f_V)$, where $f_V(t,y,z)=f(t,y+V_t,z)$.
%\end{definition}

%Let us adopt the shorthand BSDE$^{\beta}:=$BSDE$^{0,\beta}$.

\begin{remark}
Let $\beta$ be a bounded stopping time and $T:= \esssup \beta$.
Let $\hat\xi\in L^1(\mathcal F_\beta)$.  Observe that if $(Y^\beta,Z^\beta)$ is a solution to
BSDE$^\beta(\hat \xi,f)$, then $(Y^\beta_{\cdot\wedge \beta},\mathbf1_{[0,\beta]}Z^\beta)$
is a solution to BSDE$^T(\hat\xi,\mathbf 1_{[0,\beta]}f)$. Conversely, if $(Y^T,Z^T)$ is a solution
to BSDE$^T(\hat\xi,\mathbf 1_{[0,\beta]}f)$, then $(Y^T,Z^T)$
is a solution to BSDE$^\beta(\hat\xi,f)$.
\end{remark}

In light of the above remark, we may focus on BSDEs with deterministic terminal time.
In the remainder of the paper, we fix a  number $T>0$.
Let us adopt the shorthands $||\cdot||_{\mathcal{S}^p}:=||\cdot||_{\mathcal{S}^p_{\mathbb{F}}(0,T)}$,
$|\cdot|_{\mathcal{S}^p}:=|\cdot|_{\mathcal{S}^p_{\mathbb{F}}(0,T)}$,
$\|\cdot\|_{L^{r,q}_{\mathbb{F}}}:=\|\cdot\|_{L^{r,q}_{\mathbb{F}}(0,T)}$, $\|\cdot\|_{L^{r}}:=\|\cdot\|_{L^{r}(\mathcal F_T)}$,
$||\cdot||_{\mathcal{H}^p}:=||\cdot||_{\mathcal{H}^p_{\mathbb{F}}(0,T)}$, $|\cdot|_{\mathcal{H}^p}:=|\cdot|_{\mathcal{H}^p_{\mathbb{F}}(0,T)}$,
$||\cdot||_{\mathcal{D}^p}:= ||\cdot||_{\mathcal{D}^p(0,T)}$.

The results presented below have been proven in  \cite{bdh}.

\begin{theorem}\label{5grudnia3}
Let $p>1$. 
\begin{enumerate}
\item[(i)] Assume that \textnormal{(H1)--(H5)} are in force.
Then there exists a  solution $(Y,Z)\in\mathcal{S}^p_{\mathbb{F}}(0,T)\times \mathcal H^p_{\mathbb{F}}(0,T)$ to \textnormal{BSDE}$^T(\xi,f)$. 
\item[(ii)] Assume that \textnormal{(H1),(H2)} are satisfied. Then  there exists at most one solution $(Y,Z)$ to BSDE$^T(\xi,f)$
such that $Y\in\mathcal S^p_{\mathbb F}(0,T)$.
\end{enumerate}

%What is more, \textcolor{red}{$E\Big(\int^T_0|f(s,Y_s,Z_s)|\,ds\Big)^p<\infty$.}
\end{theorem}

\begin{proposition}\label{13stycznia19}
Let  $p>1$. Assume that \textnormal{(H1), (H2), (H5)} are satisfied.
Let $(Y,Z)$ be a solution to \textnormal{BSDE}$(\xi,f)$ 
such that $Y\in\mathcal{S}^p_{\mathbb{F}}(0,T)$. 
Then there exists $c_p>0$, depending only on $p$, such that 
\begin{equation}
\label{13stycznia20}
\begin{split}
\mathbb E\Big[\sup_{0\le t\le T}e^{a t}|Y_t|^p+\Big(\int^T_0 e^{2a r} |Z_r|^2\,dr\Big)^{\frac{p}{2}}\Big]\le c_p\mathbb E\Big[e^{apT}|\xi|^p+\Big(\int^T_0 e^{ar}|f(r,0,0)|\,dr\Big)^p\Big],
\end{split}
\end{equation}
for any $a\ge  \mu+\frac{\lambda ^2}{1\wedge (p-1)}$.
\end{proposition}

%\begin{corollary}\label{13stycznia32}
%Let $p>1$. Assume  \textnormal{(H1), (H2), (H5)} are satisfied. Let $(Y,Z)$ be a solution to \textnormal{BSDE}$(\xi,f)$ such that $Y\in\mathcal{S}^p_{\mathbb{F}}(0,T)$.
%Then there exists $c>0$, depending only on $\mu,\lambda,T,p$, such that
%\[
%\mathbb{E}\Big(\int^T_0 |f(r,Y_r,0)|\,dr\Big)^p\le c\mathbb{E}\Big[|\xi|^p+\Big(\int^T_0 |f(r,0,0)|\,dr\Big)^p\Big].
%\]
%\end{corollary}
%\begin{proof}
%It follows from  (H1) and Proposition \ref{13stycznia19}.
%\end{proof}

\section{Backward SDEs with $L^p$-data: the case $p=1$}

\subsection{Preliminary results}

The following result has been proven in \cite{bdh}.

\begin{theorem}\label{12sierpnia1}
Let $p=1$. Assume that \textnormal{(H1)-(H5), (Z)} are in force. 
 Then the following assertions hold.
 \begin{enumerate}
 \item[(i)]  There exists a  solution $(Y,Z)$ of \textnormal{BSDE}$^T(\xi,f)$ such that $Y$ is of class \textnormal{(D)} and $Z\in\mathcal{H}^s_{\mathbb{F}}(0,T)$, $s\in(0,1)$.
 \item[(ii)] There exists at most one solution  $(Y,Z)$ to \textnormal{BSDE}$^T(\xi,f)$ such that $Y$ is of class  \textnormal{(D)}.
 \end{enumerate} 
\end{theorem}
\begin{proof}
The assertion (i) follows from \cite[Theorem 6.3]{bdh}. As to  (ii), by \cite[Theorem 6.2]{bdh},
there exists at most one solution  $(Y,Z)$ to \textnormal{BSDE}$^T(\xi,f)$ such that $Y$ is of class  \textnormal{(D)}
and $Z\in\mathcal{H}^s_{\mathbb{F}}(0,T)$, $s\in(0,1)$. So, it is enough to show that if 
$(Y,Z)$ is a solution  to \textnormal{BSDE}$^T(\xi,f)$ such that $Y$ is of class  \textnormal{(D)}, then $Z\in\mathcal{H}^s_{\mathbb{F}}(0,T)$, $s\in(0,1)$.
This follows at once from \cite[Remark 2.1]{KRzS2} and \cite[Lemma 3.1]{bdh}.
\end{proof}

\begin{remark}
\label{rem.chv}
Let $a\in \mathbb R$. Observe that if  $(Y,Z)$ is a solution to BSDE$^T(\xi,f)$, then $(\bar Y,\bar Z)$ is a solution to BSDE$^T(\bar\xi,\bar f)$,
where
\[
(\bar Y_t,\bar Z_t):= (e^{a t} Y_t, e^{a t} Z_t),\quad \bar \xi:= e^{a T}\xi, \quad \bar f(t,y,z):= e^{a t}f(t,e^{-a t}y,e^{-a t} z)-a y
\]
Clearly, if $(\xi,f)$ satisfies any of conditions (H1), (H3)--(H5), then $(\bar\xi,\bar f)$ satisfies it too.
If $f$ satisfies (H2), then $\bar f$ satisfies (H2) but with $\mu$ replaced by $\mu-a$, and if $f$ satisfies (Z), then $\bar f$
satisfies (Z) with $(\gamma,g_t)$ replaced by $(\gamma e^{a^+ T}, g_te^{\frac{-a^-t}{\kappa}})$.
\end{remark}

We let $\mathrm{sgn}(x):=\frac{x}{|x|},\, x\in\mathbb R^k,\, x\neq 0$, and $\mathrm{sgn}(x)=0,\, x=0$.

\begin{proposition}\label{3marca1}
Let $p=1$. Assume that $f$ does not depend on $z$ and that \textnormal{(H2)} is in force. Let $(Y,Z)$ be a solution to \textnormal{BSDE}$^T(\xi,f)$ such that $Y$ is of class \textnormal{(D)}. Then for any $a\ge \mu$
\begin{equation}\label{2czerwca2}
||e^{\cdot a}Y||_{\mathcal{D}^1}\le\mathbb{E}\Big( e^{a T}|\xi|+\int^T_0 e^{a r}|f(r,0)|\,dr\Big).
\end{equation}
Moreover, $Y\in\mathcal{S}^q_{\mathbb{F}}(0,T)$, $q\in(0,1)$ and
\begin{equation}\label{2czerwca4}
|e^{\cdot a}Y|_{S^q}\le\frac{1}{1-q}\Bigg[\mathbb{E}\Big(e^{a Tq}|\xi|+\int^T_0e^{a rq}|f(r,0)|\,dr\Big)\Bigg]^q.
\end{equation}
Furthermore, if $k=1$, then
\begin{equation}\label{2czerwca2td}
\mathbb E\int_0^Te^{ar}|f(r,Y_r)|\,dr\le2\mathbb{E}\Big( e^{a T}|\xi|+\int^T_0 e^{a r}|f(r,0)|\,dr\Big).
\end{equation}
\end{proposition}
\begin{proof}
In light of Remark \ref{rem.chv}, we may assume that $\mu\le 0 $ in condition (H2) and $a=0$. 
Let us define
\[
\tau_n=\inf\big\{t\ge 0:\,\int^t_0|Z_r|^2\,dr\ge n\big\}\wedge T.
\]
By the It\^o-Tanaka  formula (see \cite[Corollary 2.3]{bdh}), for any stopping time $\sigma\le \tau_k$,
\begin{equation}\label{2czerwca1}
\begin{split}
|Y_{\sigma}|\le|Y_{\tau_n}|+\int^{\tau_n}_{\sigma}\langle \mathrm{sgn}(Y_r),f(r,Y_r)\rangle \,dr-\int^{\tau_n}_{\sigma}\langle \mathrm{sgn}(Y_r)Z_r,\,dB_r\rangle.
\end{split}
\end{equation}
By (H2)
\[
\langle \mathrm{sgn}(Y_r),f(r,Y_r)-f(r,0)\rangle\le0.
\]
From the above and  \eqref{2czerwca1} we deduce that
\begin{equation}
\begin{split}
|Y_{\sigma}|+\int_{\sigma}^{\tau_n}\Big|\langle \mathrm{sgn}(Y_r),f(r,Y_r)-f(r,0)\rangle\Big|\,dr %\int_{\sigma}^{\tau_k}|f(r,Y_r)-f(r,0)|\,dr
\le|Y_{\tau_n}|+\int^T_0|f(r,0)|\,dr-\int^{\tau_n}_{\sigma}\langle\mathrm{sgn}(Y_r)Z_r,\,dB_r\rangle.
\end{split}
\end{equation}
By the definition of $(\tau_n)$, we have that $\int^{\cdot\wedge\tau_n}_0\,\langle\mathrm{sgn}(Y_r)Z_r,dB_r\rangle$ is a martingale. Therefore
\begin{equation}\label{2czerwca3}
\begin{split}
\mathbb{E}|Y_{\sigma}|+\mathbb E\int_{\sigma}^{\tau_n}\Big|\langle \mathrm{sgn}(Y_r),f(r,Y_r)-f(r,0)\rangle\Big|\,dr\le\mathbb{E}\Big(|Y_{\tau_n}|+\int^T_0|f(r,0)|\,dr\Big),
\end{split}
\end{equation}
Passing to the limit with $n\to\infty$ we get \eqref{2czerwca2}. 
By \cite[Remark 2.1]{KRzS2} we have 
\begin{equation}
\label{eq.dq}
\mathbb E\sup_{t\le T}|Y_t|^q\le \frac{1}{1-q}\|Y\|_{\mathcal D^1}^q
\end{equation}
for any $q\in(0,1)$. This combined with \eqref{2czerwca2} yields  \eqref{2czerwca4}. 
In case $k=1$
\[
\mathbb E\int_{\sigma}^{\tau_n}\Big|\langle \mathrm{sgn}(Y_r),f(r,Y_r)-f(r,0)\rangle\Big|\,dr=
\mathbb E\int_{\sigma}^{\tau_n}|f(r,Y_r)-f(r,0)|\,dr.
\]
Therefore, letting $n\to\infty$ in \eqref{2czerwca3} gives 
\begin{equation}\label{2czerwca3td}
\begin{split}
\mathbb E \int_0^T|f(r,Y_r)-f(r,0)|\,dr\le \mathbb{E}\Big(|\xi|+\int^T_0|f(r,0)|\,dr\Big).
\end{split}
\end{equation}
From this we easily conclude  \eqref{2czerwca2td}.
\end{proof}

\subsection{Main results}
 We shall adopt the following notation: $e_a:\mathbb R\to\mathbb R$, with $e_a(x):= e^{ax},\, x\in\mathbb R$.
 We will use frequently the following function
 \[
 \mathcal C(x,y,z):= 2(1+y)x[(z+\lambda)\vee1](1\vee T)^{3},\quad x,y,z\ge 0.
 \]
 Observe that for any $x,y_1,y_2,z\ge 0$
 %\begin{equation}
% \label{eq.sub1}
%  \mathcal C(x,y_1+y_2,z)\le   \mathcal C(x,y_1,z)+  y_2\mathcal C(x,y_1,z),
% \end{equation}
%and
 \begin{equation}
 \label{eq.sub2}
  \mathcal C(x,y_1+y_2,z)\le   \mathcal C(x,y_1,z)+  \mathcal C(x,y_2,z).
 \end{equation}

\begin{theorem}
\label{th.1}
Let $p=1$. Assume  \textnormal{(H1)--(H5), (Z)}.  
\begin{enumerate}
\item[(i)] For any $q\in (\kappa,1)$ there exists $c_{\kappa,q}>0$ - depending only on $q,\kappa$ - 
such that  for any solution  $(Y,Z)$ to \textnormal{BSDE}$(\xi,f)$ such that $Y$  is of class \textnormal{(D)}
and any $a\ge \mu+\frac{\lambda^2}{1\wedge (\frac q\kappa -1 )}$, we have
\begin{equation}
\label{eq.16.01.1}
\begin{split}
&\|e_aY\|_{\mathcal{D}^1}+\mathbb E\Big(\int^T_0 e^{2ar} |Z_r|^2\,dr\Big)^{\frac{q}{2}}
\\&\quad\quad\le \mathcal C\big(c_{\kappa,q},\|e_{-\frac{a^-}{\kappa}} g\|_{L^1_{\mathbb F}}, e^{a^+T}\gamma\big)\psi_1\Big(e^{aT}\mathbb E|\xi|+\mathbb E\int^T_0 e^{ar}|f(r,0,0)|\,dr\Big),
\end{split}
\end{equation}
where $\psi_1(x)=x+x^{\kappa^2(1-q)},\, x\ge 0$.
\item[(ii)] Let $q\in (\kappa,1)$ and $k=1$.
Then for any solution  $(Y,Z)$ to \textnormal{BSDE}$(\xi,f)$, such that $Y$  is of class \textnormal{(D)},
and any $a\ge \mu+\frac{\lambda^2}{1\wedge (\frac q\kappa -1 )}$, we have
\begin{equation}
\label{eq.16.01.1ff}
\begin{split}
\mathbb E\int^T_0 &e^{ar}|f(r,Y_r,Z_r)|\,dr
\\&\le \mathcal C^2\big(c_{\kappa,q},\|e_{- \frac{a^-}{\kappa}} g\|_{L^1_{\mathbb F}}, e^{a^+T}\gamma\big)
 \psi_2\Big(e^{a T}\mathbb E|\xi|+\mathbb E\int^T_0 e^{a r}|f(r,0,0)|\,dr\Big),
\end{split}
\end{equation}
 with $c_{\kappa,q}$ as in (i),  and  $\psi_2(x)=x+x^{\kappa^3 (1-q)^2},\,  x\ge 0.$
\end{enumerate}

\end{theorem}
\begin{proof}
By Remark \ref{rem.chv}, we may assume that 
 $0=a\ge  \mu+\frac{\lambda ^2}{1\wedge(\frac q\kappa-1)}$.
Throughout the proof %$c_1,c_2,...$ etc.  denote  constants greater than or equal to $1$ depending on $\mu^+,\lambda,T,q,\kappa,\gamma$,
%and if we want to use  a positive constant, which depends on the different set of parameters, say 
$c_{\gamma_1,...,\gamma_k}$ denote a constant, which may vary from line to line, but it   depends  only  on parameters $\gamma_1,\gamma_2,...,\gamma_k$.
  Throughout the proof, we frequently use  the following elementary inequality
\begin{equation}
\label{eq.el1}
x^b\le x+x^{a},\quad x\ge 0, \, 0\le a\le b\le 1.
\end{equation}
Fix $q\in (\kappa,1)$. Set $p_0:=\frac{q}{\kappa}>1$.
We let $f^0(t,y):= f(t,y,0)$.
Let $(Y^0,Z^0)$ be a solution to BSDE$^T(\xi,f^0)$, such that $Y^0$ is of class (D) and $Z^0\in \mathcal{H}^s_{\mathbb F}(0,T),\, s\in (0,1)$
(see Theorem \ref{12sierpnia1}).  Observe that the pair $(\bar Y,\bar Z):= (Y-Y^0,Z-Z^0)$ is a solution to 
BSDE$^T(0,F)$ with
\[
F(t,y,z):= f(t,y+Y^0_t,z+Z^0_t)-f(t,Y^0_t,0).
\]
It is an elementary check that $F$ satisfies conditions (H1)--(H3), (Z) and  (H5) for any $p\ge 1$.  
By \cite[Corollary 2.3]{bdh}
\begin{align*}
|\bar Y_t|\le \mathbb E\Big(\int_0^T|F(r,0,\bar Z_r)|\,dr\Big|\mathcal F_t\Big)&=\mathbb E\Big(\int_0^T|f(r,Y^0_r, Z_r)-f(r,Y^0_r,0)|\,dr\Big|\mathcal F_t\Big) 
\\&\le \gamma\mathbb E\Big(\int_0^T(g_r+|Y^0_r|+|Z_r|)^\kappa\,dr\Big|\mathcal F_t\Big).
\end{align*}
Therefore, by  Doob's inequality,
\[
\mathbb E\sup_{t\le T}|\bar Y_t|^{p_0}\le \gamma c_pT^{p_0-1} \mathbb E\int_0^T(g_r+|Y^0_r|+|Z_r|)^{q}\,dr<\infty
\]
Consequently, by \cite[Lemma 3.1]{bdh} $\bar Z\in\mathcal H_{\mathbb F}^{p_0}(0,T)$.
%Moreover, by Lemma \ref{eq.eleq}, we also have (H4) satisfied by $F$. Consequently, by Theorem \ref{12sierpnia1} and Theorem \ref{5grudnia3},
%$(\bar Y,\bar Z)\in\mathcal S^{p_0}_{\mathbb F}(0,T)\times \mathcal H_{\mathbb F}^{p_0}(0,T)$.
Observe that conditions (H1) and (Z) (for $f$) together imply that 
\[
|F(t,0,0)|\le (\lambda+\gamma)(g_t\mathbf1_{\{|Z^0_t|\ge 1\}}+|Y^0_t|+|Z^0_t|)^\kappa,\quad t\in [0,T].
\] 
Therefore,  by Proposition \ref{13stycznia19} and \cite[Lemma 3.1]{bdh},
\begin{equation}
\label{eq.mmm1}
\begin{split}
&\mathbb E\sup_{t\le T}|\bar Y_t|^{p_0}+\mathbb E\Big(\int_0^T|\bar Z_r|^2\,dr\Big)^{p_0/2}\le c_{p_0}\mathbb E\Big(\int_0^T|F(r,0,0)|\,dr\Big)^{p_0}
\\&
\le c_{p_0}(\gamma+\lambda) \mathbb E\Big(\int_0^T(|g_r|\mathbf1_{\{|Z^0_r|\ge 1\}}+|Y^0_r|+|Z^0_r|)^\kappa\,dr\Big)^{p_0}
\\&
\le 3^{p_0} c_{p_0}(\gamma+\lambda)\Big[\mathbb E \Big(\int_0^T  g^\kappa_r\mathbf1_{\{|Z^0_r|\ge 1\}}\,dr\Big)^{p_0}
+T^{p_0}\mathbb E\sup_{t\le T} |Y^0_t|^{q}+T^{\frac{p_0(2-\kappa)}{2}}\mathbb E\Big(\int_0^T|Z^0_r|^2\,dr\Big)^{\frac q2}\Big]
\\&
\le 3^{p_0}(1\vee T)^{p_0}c_{p_0}(\gamma+\lambda)\Big[ \mathbb E\Big(\int_0^T g^\kappa_r\mathbf1_{\{|Z^0_r|\ge 1\}}\,dr\Big)^{p_0}+\mathbb E\sup_{t\le T}|Y^0_t|^{q}
+\mathbb E\Big(\int_0^T|Z^0_r|^2\,dr\Big)^{\frac q2}\Big]
\\&
\le 3^{p_0}(1\vee T)^{p_0}c_{p_0}(\gamma+\lambda)\Big[ \mathbb E\Big(\int_0^T g^\kappa_r\mathbf1_{\{|Z^0_r|\ge 1\}}\,dr\Big)^{p_0}
+c_q\Big(\mathbb E\int_0^T|f(r,0,0)|\,dr+\mathbb E|\xi|\Big)^q\Big] .
\end{split}
\end{equation}
For brevity, we let  $T_1:=1\vee T$, $\gamma_\lambda:=(\lambda+\gamma)\vee 1$ and $A:=3^{p_0}T_1^{p_0}c_{p_0}\gamma_\lambda$.
Now, we shall estimate the first term on the right-hand side of \eqref{eq.mmm1}. 
By using H\"older's inequality, we compute that
%\begin{equation}
%\label{eq.split1}
%\begin{split}
%\mathbb E\Big(\int_0^T g^\kappa_r\mathbf1_{\{|Z^0_r|\ge 1\}}\,dr\Big)^p&
%\le \mathbb E\Big[\Big(\int_0^Tg_r^{\kappa\beta}\,dr\Big)^{p/\beta}
%\Big(\int_0^T\mathbf1_{\{|Z^0_r|\ge 1\}}\,dr\Big)^{p/\beta^*}\Big]\\&
%\le \Big[\mathbb E\Big( \int_0^Tg_r^{\kappa\beta}\,dr\Big)^{\frac{p\theta}{\beta}}\Big]^{1/\theta}
%\Big[\mathbb E\Big(\int_0^T\mathbf1_{\{|Z^0_r|\ge 1\}}\,dr\Big)^{\frac{p\theta^*}{\beta^*}}\Big]^{1/\theta^*}.
%\end{split}
%\end{equation}
%Letting $\beta:=1/\kappa$ and $\theta:=\frac{1}{q}$, we find that 
\begin{equation}
\label{eq.split2}
\begin{split}
\mathbb E\Big(\int_0^T g^\kappa_r\mathbf1_{\{|Z^0_r|\ge 1\}}\,dr\Big)^{p_0}
&\le T^{p_0-1}\mathbb E\Big(\int_0^T g^q_r\mathbf1_{\{|Z^0_r|\ge 1\}}\,dr\Big)
\\&\le T^{p_0-1}\Big(\mathbb E\int_0^Tg_r\,dr\Big)^{q} \Big[\mathbb E\Big(\int_0^T\mathbf1_{\{|Z^0_r|\ge 1\}}\,dr\Big)\Big]^{1-q}.
\end{split}
\end{equation}
%Let $a>1$ and $b:=\frac{a(p-q)}{1-q}>1$.
%Utilizing the fact that function $x\mapsto x^b$ is convex  on $[0,\infty)$, we compute that
By H\"older's inequality again
\begin{equation}
\label{eq.split3}
\begin{split}
\Big[\mathbb E\Big(\int_0^T\mathbf1_{\{|Z^0_r|\ge 1\}}\,dr\Big)\Big]^{1-q}
&=\Big[\mathbb E\Big(\int_0^T\mathbf1_{\{|Z^0_r|^q\ge 1\}}\,dr\Big)\Big]^{1-q}
\\&\le
\Big[\mathbb E\int_0^T|Z^0_r|^q\,dr\Big]^{1-q}\le T^{\frac{(2-q)(1-q)}{2}}\Big[\Big(\mathbb E\int_0^T|Z^0_r|^2\,dr\Big)^{\frac q2}\Big]^{1-q}.
\end{split}
\end{equation}
%Letting   $a:= 2/q$, we obtain  
%\begin{equation}
%\label{eq.split4}
%\begin{split}
%\Big[\mathbb E\Big(\int_0^T\mathbf1_{\{|Z^0_r|\ge 1\}}\,dr\Big)^{\frac{p-q}{1-q}}\Big]^{1-q}
%\le T^{(b-1)(1-q)}\Big[\mathbb E\Big(\int_0^T|Z^0_r|^2\,dr\Big)^{q/2}\Big]^{1-q}.
%\end{split}
%\end{equation}
Applying, respectively,   \cite[Lemma 3.1]{bdh}, Proposition \ref{3marca1} and  Jensen's inequality, we find that
\begin{equation}
\label{eq.split5}
\begin{split}
\Big[\mathbb E&\Big(\int_0^T|Z^0_r|^2\,dr\Big)^{\frac q2}\Big]^{1-q}\le 
c_q\Big[\mathbb E\sup_{t\le T}|Y^0_t|^q+\mathbb E\Big(\int_0^T|f(r,0,0)|\,dr\Big)^q\Big]^{1-q}
\\&
\le c_q\Big[\frac{1}{1-q}\Big(\mathbb E|\xi|+\mathbb E\int_0^T|f(r,0,0)|\,dr\Big)^q+\mathbb E\Big(\int_0^T|f(r,0,0)|\,dr\Big)^q\Big]^{1-q}
\\&
\le c_q\Big(\mathbb E\int_0^T|f(r,0,0)|\,dr+\mathbb E|\xi|\Big)^{q(1-q)}.
\end{split}
\end{equation}
Set %$B:= c_q$, $C:= \textcolor{red}{\big(\frac{1}{1-q}+1\big)(c_q+1)}$, and 
$K:=\mathbb E\int_0^T|f(r,0,0)|\,dr+\mathbb E|\xi|$. Combining %\eqref{eq.el1}, 
\eqref{eq.mmm1}--\eqref{eq.split5}
implies that
\begin{equation}
\label{eq.split6}
\|\bar Y\|^{p_0}_{\mathcal S^{p_0}}+\|\bar Z\|^{p_0}_{\mathcal H^{p_0}}
\le c_qA\Big(T_1^{p_0}\|g\|^q_{L^1_{\mathbb F}}K^{q(1-q)}+K^q\Big).
\end{equation}
Hence
\[
\|\bar Y\|^{p_0}_{\mathcal{D}^1}\le c_qA\Big(T_1^{p_0}\|g\|^q_{L^1_{\mathbb F}}K^{q(1-q)}+K^q\Big),\quad |\bar Z|^{p_0/q}_{\mathcal H^q}\le  c_qA\Big(T_1^{p_0}\|g\|^q_{L^1_{\mathbb F}}K^{q(1-q)}+K^q\Big).
\]
Consequently,
\[
\| Y\|_{\mathcal{D}^1}\le (c_qA)^{1/p_0}\Big(T_1^{p_0}\|g\|^q_{L^1_{\mathbb F}}K^{q(1-q)}+K^q\Big)^{1/p_0}+\|Y^0\|_{\mathcal{D}^1},
\]
\[
| Z|_{\mathcal H^q}\le  (c_qA)^{q/p_0}\Big(T_1^{p_0}\|g\|^q_{L^1_{\mathbb F}}K^{q(1-q)}+K^q\Big)^{q/p_0}+|Z^0|_{\mathcal H^q}.
\]
By Proposition \ref{3marca1}, \cite[Lemma 3.1]{bdh} and \eqref{eq.el1},
\begin{equation}
\label{eq.split8aa}
\begin{split}
\|Y\|_{\mathcal{D}^1}&\le (c_qA)^{1/p_0}\Big(T_1^{p_0}\|g\|^q_{L^1_{\mathbb F}}K^{q(1-q)}+K^q\Big)^{1/p_0}+K
\\&
\le (1+\|g\|_{L^1_{\mathbb F}})c_{\kappa,q}\gamma_\lambda T_1^{3}(K^{\kappa^2(1-q)}+K)
=\frac12 \mathcal C\big(c_{\kappa,q},\|g\|_{L^1_{\mathbb F}},\gamma\big)(K^{\kappa^2(1-q)}+K),
\end{split}
\end{equation}
and
\begin{equation}
\label{eq.split8bb}
\begin{split}
|Z|_{\mathcal H^q}&\le (c_qA)^{q/p_0}\Big(T_1^{p_0}\|g\|^q_{L^1_{\mathbb F}}K^{q(1-q)}+K^q\Big)^{q/p_0} +c_qK^q
\\&
\le  (1+\|g\|_{L^1_{\mathbb F}})c_{\kappa,q}\gamma_\lambda T_1^{3}(K^{\kappa^2(1-q)}+K)
=\frac12\mathcal C\big(c_{\kappa,q},\|g\|_{L^1_{\mathbb F}},\gamma\big)(K^{\kappa^2(1-q)}+K).
\end{split}
\end{equation}
We see that \eqref{eq.split8aa}, \eqref{eq.split8bb} imply \eqref{eq.16.01.1}.
%From these two inequalities, we would conclude the asserted inequality but without the  term $\mathbb E \int_0^T|f(r,Y_r,Z_r)|\,dr$
%on the left-hand side  of \eqref{eq.16.01.1}.  
In order to obtain \eqref{eq.16.01.1ff}, we look at $(Y,Z)$ as a solution
to BSDE$^T(\xi,f_Z)$, where $f_Z(t,y):= f(t,y,Z_t)$, in other words, we freeze  $Z$ in the driver $f$. From this perspective  $(Y,Z)$
is a solution to BSDE with the driver independent of $z$ variable. Therefore, by \eqref{2czerwca2td}
applied to $f_Z$, we have
\begin{equation}
\label{eq.split7}
\mathbb E\int_0^T|f(r,Y_r,Z_r)|\,dr\le 2\Big(\mathbb E|\xi|+\mathbb E\int_0^T|f(r,0,Z_r)|\,dr\Big).
\end{equation}
By (Z) and (H1)
\begin{equation}
\label{eq.split8}
\mathbb E\int_0^T|f(r,0,Z_r)|\,dr\le \gamma_\lambda \mathbb E\int_0^T(g_r^\kappa\mathbf1_{\{|Z_r|\ge 1\}}+|Z_r|^\kappa)\,dr
+\mathbb E\int_0^T|f(r,0,0)|\,dr.
\end{equation}
By H\"older's inequality 
\begin{equation}
\label{eq.split9}
\mathbb E\int_0^Tg_r^\kappa\mathbf1_{\{|Z_r|\ge 1\}}\,dr\le 
\|g\|_{L^1_{\mathbb F}}^\kappa\Big(\mathbb E\int_0^T\mathbf1_{\{|Z_r|\ge 1\}}\,dr\Big)^{1-\kappa},
\end{equation}
\begin{equation}
\label{eq.split10}
\Big(\mathbb E\int_0^T\mathbf1_{\{|Z_r|\ge 1\}}\,dr\Big)^{1-\kappa}\le \Big(\mathbb E\int_0^T|Z_r|^q\,dr\Big)^{1-\kappa}
\le  \Big[ T^{1-q/2} \mathbb E\Big(\int_0^T|Z_r|^2\,dr\Big)^{q/2}\Big]^{1-\kappa},
\end{equation}
and
\begin{equation}
\label{eq.split11}
 \mathbb E\int_0^T|Z_r|^\kappa\,dr
\le   T^{1-\frac{\kappa}{2}} \Big[\mathbb E\Big(\int_0^T|Z_r|^2\,dr\Big)^{q/2}\Big]^{\kappa/q}.
\end{equation}
By virtue of %\eqref{eq.split8bb} and 
\eqref{eq.split7}--\eqref{eq.split11},
we conclude that
\begin{equation}
\label{eq.split55td}
\begin{split}
\mathbb E\int_0^T|f(r,Y_r,Z_r)|\,dr&\le 2\Big[K+  \gamma_\lambda \|g\|_{L^1_{\mathbb F}}^\kappa\Big[ T^{1-q/2} |Z|_{\mathcal H^q}\Big]^{1-\kappa}
+\gamma_\lambda T^{1-\frac\kappa2} |Z|_{\mathcal H^q}^{\kappa/q} \Big]
\\& \le 2\Big[K+  \gamma_\lambda \|g\|_{L^1_{\mathbb F}}^\kappa T_1^2 |Z|_{\mathcal H^q}^{1-\kappa}+\gamma_\lambda T_1 |Z|_{\mathcal H^q}^{\kappa/q} \Big].
\end{split}
\end{equation}
This combined with \eqref{eq.split8bb} gives, with the shorthand $\mathcal C=\mathcal C\big(c_{\kappa,q},\|g\|_{L^1_{\mathbb F}},\gamma\big)$,
\begin{equation}
\label{eq.split55}
\begin{split}
\mathbb E\int_0^T|f(r,Y_r,Z_r)|\,dr&\le 2\Big[K+  T_1^2\gamma_\lambda \mathcal C^{1-\kappa}(K^{\kappa q(1-q)(1-\kappa)}+K^{q(1-\kappa)})
\\&\quad
+\gamma_\lambda T_1\mathcal C^{\kappa/q}(K^{\kappa^2 (1-q)}+K^{\kappa}) \Big].
\end{split}
\end{equation}
Observe that $\kappa^3(1-q)^2$ is smaller than any exponent of a power with base $K$
on the right-hand side of \eqref{eq.split55}.%, \eqref{eq.split8bb} or  \eqref{eq.split55}.
 Therefore, from  \eqref{eq.el1}  and  \eqref{eq.split55}, we conclude that 
\begin{equation}
\label{eq.16.01.1uu}
\begin{split}
\mathbb E\int^T_0 |f(r,Y_r,Z_r)|\,dr
\le \mathcal C^2\Big(K+K^{\kappa^3(1-q)^2}\Big).
\end{split}
\end{equation}
This completes the proof.
 %Observe that condition (Z) also holds with $g_t$ replaced by $1+g_t$ and $\kappa$ replaced by $\kappa\vee \frac12$.
% Making these substitutions in \eqref{eq.split8aa}, \eqref{eq.split8bb} and \eqref{eq.16.01.1uu} and using once again \eqref{eq.el1} lead us to \eqref{eq.16.01.1} and \eqref{eq.16.01.1ff} for $q\in (\kappa\vee\frac12,1)$. 
\end{proof}

\begin{theorem}
\label{th.2}
Let $p=1$. Consider a function  $\bar f:\Omega\times [0,T]\times\mathbb R^k\times\mathbb R^{d\times k}\to \mathbb R^k$
and $\bar\xi\in L^1(\mathcal F_T)$.
Assume that \textnormal{(H1)-(H5), (Z)} are in force. 
Let $(Y,Z), (\bar Y,\bar Z)$ be solutions 
to \textnormal{BSDE}$^T(\xi,f)$, \textnormal{BSDE}$^T(\bar\xi,\bar f)$, respectively, 
such that $Y, \bar Y$ are of class \textnormal{(D)}. % and $\bar Z\in \mathcal H_{\mathbb F}^s(0,T),\, s\in (0,T)$.
Suppose that $\mathbb E\int_0^T|f(r,\bar Y_r,\bar Z_r)-\bar f(r,\bar Y_r,\bar Z_r)|\,dr<\infty$. 
\begin{enumerate}
\item[(i)]
For any $q\in (\kappa,1)$
there exists $c_{\kappa,q}>0$ - depending only on $\kappa,q$ - such that for any $a\ge \mu+\frac{\lambda^2}{1\wedge(\sqrt{\frac q\kappa}-1)}$,
\begin{equation}
\label{eq.split75}
\begin{split}
\|e_a(Y-\bar Y)\|_{\mathcal D^1}&+|e_a(Z-\bar Z)|_{\mathcal H^q}
\\&
\le 
\mathcal C\big(c_{\kappa,q},\|e_{- \frac{a^-}{\hat\kappa}}\hat g\|_{L^1_{\mathbb F}},2e^{a^+T}\gamma\big)
\psi_3(\|\xi-\bar\xi\|_{L^1}+\||f-\bar f|(\cdot,\bar Y,\bar Z)\|_{L^1_{\mathbb F}}),
\end{split}
\end{equation}
where $\hat{g}_t:=g_t+3+|\bar{Y}_t|^{\frac{\kappa}{\hat\kappa}}+|\bar{Z}_t|^{\frac{\kappa}{\hat\kappa}}$, $\hat\kappa :=\sqrt{\kappa q}$, and 
$\psi_3(x)=x+x^{\hat \kappa^2(1-q)},\, x\ge 0$.
\item[(ii)] Let $q\in (\kappa,1)$ and $k=1$. Then for  any  $a\ge \mu+\frac{\lambda^2}{1\wedge(\sqrt{\frac q\kappa}-1)}$,
\begin{equation}
\label{eq.split75td}
\begin{split}
\mathbb E\int_0^Te^{ar}&|f(r,Y_r,Z_r)-\bar f(r,\bar Y_r,\bar Z_r)|\,dr
\\&
\le 
\mathcal C^2\big(c_{\kappa,q},\|e_{- \frac{a^-}{\hat\kappa}}\hat g\|_{L^1_{\mathbb F}},2e^{a^+T}\gamma\big)
\psi_4(\|\xi-\bar\xi\|_{L^1}+\||f-\bar f|(\cdot,\bar Y,\bar Z)\|_{L^1_{\mathbb F}}),
\end{split}
\end{equation}
with $c_{\kappa,q}, \hat\kappa,\hat g$ as in (i), and $\psi_4(x):=x+x^{\hat\kappa^3 (1-q)^2} $.
\end{enumerate}
\end{theorem}

\begin{proof}
Set $\mathcal C:=\mathcal C\big(c_{\kappa,q},\|e_{- \frac{a^-}{\hat\kappa}}\hat g\|_{L^1_{\mathbb F}},2e^{a^+T}\gamma\big)$.
Let $(Y,Z), (\bar Y,\bar Z)$ be as in the assertion of the theorem. 
By Theorem \ref{12sierpnia1} $Z\in \mathcal H_{\mathbb F}^s(0,T),\, s\in (0,T)$.
Observe  that
\[
Y_t-\bar{Y}_t=\xi-\bar\xi+\int^T_0 F(r,Y_r-\bar{Y}_r,Z_r-\bar{Z}_r)\,dr-\int^T_t (Z_r-\bar{Z}_r)\,dB_r,\quad t\in[0,T],
\]
where $F(t,y,z)=f(t,y+\bar{Y}_t,z+\bar{Z}_t)-\bar f(t,\bar{Y}_t,\bar{Z}_t)$. In other words $(Y-\bar Y,Z-\bar Z)$
is a solution to BSDE$^T(\xi-\bar \xi,F)$.
By the assumptions made on $\bar f$, process $\bar f(\cdot,\bar Y,\bar Z)$ is $\mathbb F$-progressively measurable.
Obviously, $F$ satisfies (H1),(H2), with the same constants, and (H3),(H5).  
By  Theorem \ref{12sierpnia1} $Z-\bar Z\in\mathcal H_{\mathbb F}^s(0,T),\, s\in (0,1)$.
As a result, $\bar Z\in\mathcal H_{\mathbb F}^s(0,T),\, s\in (0,1)$. Now, by (Z), we have 
\begin{equation*}
|F(t,y,z)-F(t,y,0)|=|f(t,y+\bar{Y}_t,z+\bar{Z}_t)- f(t,y+\bar{Y}_t,\bar{Z}_t)|\le 2\gamma (g_t+|\bar{Y}_t|+|\bar{Z}_t|+|y|+|z|)^{\kappa}.
\end{equation*}
Let us take $\beta\in(\kappa,q)$. Then
\begin{equation*}
\begin{split}
&(g_t+|\bar{Y}_t|+|\bar{Z}_t|+|y|+|z|)^{\kappa}=\big((g_t+|\bar{Y}_t|+|\bar{Z}_t|+|y|+|z|)^{\frac{\kappa}{\beta}}\big)^{\beta}\\
&\le \big ((g_t\vee1)+|\bar{Y}_t|^{\frac{\kappa}{\beta}}+|\bar{Z}_t|^{\frac{\kappa}{\beta}}+(|y|\vee1)
+(|z|\vee 1)\big)^{\beta}
\le (g_t+3+|\bar{Y}_t|^{\frac{\kappa}{\beta}}+|\bar{Z}_t|^{\frac{\kappa}{\beta}}+|y|+|z|)^{\beta}.
\end{split}
\end{equation*}
Let us define $\hat{g}_t:=g_t+3+|\bar{Y}_t|^{\frac{\kappa}{\beta}}+|\bar{Z}_t|^{\frac{\kappa}{\beta}}$. 
By the fact that $\bar{Y}$ is of class (D) and  $\bar{Z}\in\mathcal{H}^s_{\mathbb{F}}(0,T)$, $s\in(0,1)$, 
we have    $\hat{g}\in L^1_{\mathbb{F}}(0,T)$. Thus, $F$ satisfies (Z), with $g$ replaced by 
$\hat{g}$, $\gamma$ replaced by $2\gamma$ and $\kappa$ replaced by $\beta$.

{\bf Step 1.} Suppose that $F$ satisfies (H4).
Letting $\beta=\sqrt{\kappa q}$ and applying  Theorem \ref{12sierpnia1} give the desired inequalities.
{\bf Step 2.} The generale case.
By the very definition of a solution to BSDE$^T(\bar\xi,\bar f)$,
\begin{equation}\label{13lipca1}
\int^T_0|\bar f(r,\bar{Y}_r,\bar{Z}_r)|\,dr<\infty.
\end{equation}
Thus, the sequence
\[
\sigma_n:=\inf\{t>0: \int^t_0|\bar f(r,\bar{Y}_r,\bar{Z}_r)|\,dr\ge n\}\wedge T,\quad n\ge 1,
\]
satisfies $P(\exists_{n\ge 1}\,\, \sigma_n=T)=1$. Let $\alpha_n:=\inf\{t>0: |\bar Y_t-\bar Y_0|\ge n\}\wedge T$.
Since $\bar Y$ is c\`adl\`ag, we also have $P(\exists_{n\ge 1}\,\, \alpha_n=T)=1$. Let $\tau_n:=\sigma_n\wedge \alpha_n$,
and $a_0:=\mathbb E|\bar Y_0|$. Then, by (Z),
\begin{align*}
&\int^{\tau_n}_0 \sup_{|y|\le M}|F(t,y,0)|\,dt
\\&
\le
\int^{\tau_n}_0 \sup_{|y|\le M}\Big(|f(t,y+\bar Y_y,\bar Z_t)-f(t,y+\bar Y_t,0)|+|f(t,y+\bar Y_t,0)|+|\bar f(t,\bar Y_t,\bar Z_t)|\Big)\,dt
\\&
\le
\gamma \int_0^T(g_t+|\bar Y_t|+|\bar Z_t|+M)^\kappa\,dt+\int_0^T\sup_{|y|\le M+n+a_0}|f(t,y,0)|\,dt+n.
\end{align*}
In consequence, $F$ satisfies  (H4) with $f$ replaced by $F$ and $\tau$ replaced by $\tau_n$
(recall that $\bar Z\in\mathcal H_{\mathbb F}^s(0,T),\, s\in (0,1)$). 
Clearly, $(Y-\bar Y,Z-\bar Z)$ is a solution to BSDE$^{\tau_n}(Y_{\tau_n}-\bar Y_{\tau_n},F)$.
By Step 2,
\begin{equation*}
\begin{split}
\|e_a(Y-\bar Y)\|_{\mathcal D^1(0,\tau_n)}+|e_a(Z-\bar Z)|_{\mathcal H^q(0,\tau_n)}&\le 
\mathcal C\psi_3(\|Y_{\tau_n}-\bar Y_{\tau_n}\|_{L^1}+\||f-\bar f|(\cdot,\bar Y,\bar Z)\|_{L^1_{\mathbb F}}),
\end{split}
\end{equation*}
and 
\begin{equation*}
\begin{split}
\mathbb E\int_0^{\tau_n}e^{ar}|f(r,Y_r,Z_r)-\bar f(r,\bar Y_r,\bar Z_r)|\,dr
&\le 
\mathcal C^2\psi_4(\|Y_{\tau_n}-\bar Y_{\tau_n}\|_{L^1}+\||f-\bar f|(\cdot,\bar Y,\bar Z)\|_{L^1_{\mathbb F}}).
\end{split}
\end{equation*}
By sending  $n\to \infty$ and using the fact that $Y, \bar Y$ are of class (D) and  $P(\exists_{n\ge 1}\,\, \tau_n=T)=1$,
we conclude the result.
\end{proof}

\begin{corollary}
\label{cor.1}
Let $p=1$. Let  $\bar f:\Omega\times [0,T]\times\mathbb R^k\times\mathbb R^{d\times k}\to \mathbb R^k$
and $\bar\xi\in L^1(\mathcal F_T)$.
Assume that $(\xi,f)$ and $(\bar\xi,\bar f)$
satisfy  \textnormal{(H1)-(H5), (Z)}. 
Let $(Y,Z), (\bar Y,\bar Z)$ be solutions 
to \textnormal{BSDE}$^T(\xi,f)$, \textnormal{BSDE}$^T(\bar\xi,\bar f)$, respectively, 
such that $Y, \bar Y$ are of class \textnormal{(D)}. % and $\bar Z\in \mathcal H_{\mathbb F}^s(0,T),\, s\in (0,T)$.
Suppose that $\mathbb E\int_0^T|f(r,\bar Y_r,\bar Z_r)-\bar f(r,\bar Y_r,\bar Z_r)|\,dr<\infty$. Let
\[
K_a:=e^{aT}\mathbb E|\bar\xi|+\mathbb E\int^T_0 e^{ar}|\bar f(r,0,0)|\,dr,\quad \delta f:=f-\bar f,\quad \delta\xi:=\xi-\bar \xi,
\]
and
\begin{align*}
L_a(x)&:=\mathcal C\big(x,\|e_{-\frac{a^-}{\kappa}} g\|_{L^1_{\mathbb F}},e^{a^+T}\gamma\big)\psi_1(K_a),\quad x\ge 0,
\end{align*}
with  $\psi_1$ is as in \textnormal{Theorem \ref{th.1}}(i).
\begin{enumerate}
\item[(i)]
For any $q\in (\kappa,1)$
there exists $c_{\kappa,q}>0$ - depending only on $\kappa,q$ - such that for any $a\ge \mu+\frac{\lambda^2}{1\wedge(\sqrt{\frac q\kappa}-1)}$,
\begin{equation}
\label{eq.split75cor}
\begin{split}
\|e_a(Y-\bar Y)\|_{\mathcal D^1}&+|e_a(Z-\bar Z)|_{\mathcal H^q}
\le T_1L^2_a(c_{\kappa,q})\psi_3(\|\delta\xi\|_{L^1}+\||\delta f|(\cdot,\bar Y,\bar Z)\|_{L^1_{\mathbb F}}),
\end{split}
\end{equation}
where $\hat\kappa:=\sqrt{\kappa q}$, and $\psi_3(x)=x+x^{\hat \kappa^2(1-q)},\, x\ge 0$.
\item[(ii)] Let $q\in (\kappa,1)$ and $k=1$. Then for  any  $a\ge \mu+\frac{\lambda^2}{1\wedge(\sqrt{\frac q\kappa}-1)}$,
\begin{equation}
\label{eq.split75tdcor}
\begin{split}
\mathbb E\int_0^Te^{ar}|f(r,Y_r,Z_r)-\bar f(r,\bar Y_r,\bar Z_r)|\,dr&
\le T_1L^3_a(c_{\kappa,q})
\psi_4(\|\delta \xi\|_{L^1}+\||\delta f|(\cdot,\bar Y,\bar Z)\|_{L^1_{\mathbb F}}),
\end{split}
\end{equation}
with $c_{\kappa,q}, \hat\kappa$ as in (i), and $\psi_4(x):=x+x^{\hat\kappa^3 (1-q)^2} $.
\end{enumerate}
\end{corollary}
\begin{proof}
Clearly, $a\ge  \mu+\frac{\lambda^2}{1\wedge(\frac q\kappa-1)}$. By using Theorem \ref{th.1}, we have 
\begin{equation*}
\begin{split}
\|e_{- \frac{a^-}{\hat\kappa}}(g+3+|\bar{Y}|^{\frac{\kappa}{\hat\kappa}}&+|\bar{Z}|^{\frac{\kappa}{\hat\kappa}})\|_{L^1_{\mathbb F}}
\le
\|e_{- \frac{a^-}{\hat\kappa}} g\|_{L^1_{\mathbb F}}+4\|e_{-a^-}\|_{L^1_{\mathbb F}}+T\|e_a \bar Y\|_{\mathcal D^1}+T_1|e_a \bar Z|_{\mathcal H^{\sqrt{\kappa/q}}}
\\&\le
\|e_{- \frac{a^-}{\hat\kappa}} g\|_{L^1_{\mathbb F}}+4(\frac{1}{a^-}\wedge T)
+T_1\mathcal C\big(c_{\kappa,q},\|e_{-\frac{a^-}{\kappa}} g\|_{L^1_{\mathbb F}},e^{a^+T}\gamma\big)\psi_1(K_a)
\\& \le 5T_1\mathcal C\big(c_{\kappa,q},\|e_{-\frac{a^-}{\kappa}} g\|_{L^1_{\mathbb F}},e^{a^+T}\gamma\big)\psi_1(K_a).
\end{split}
\end{equation*}
Therefore, by \eqref{eq.sub2}, 
\begin{equation}
\begin{split}
&\mathcal C\big(c_{\kappa,q},\|e_{- \frac{a^-}{\hat\kappa}}
(g+3+|\bar{Y}|^{\frac{\kappa}{\hat\kappa}}+|\bar{Z}|^{\frac{\kappa}{\hat\kappa}})\|_{L^1_{\mathbb F}},2e^{a^+T}\gamma\big)
\le
T_1\mathcal C^2\big(c_{\kappa,q},\|e_{-\frac{a^-}{\kappa}} g\|_{L^1_{\mathbb F}},e^{a^+T}\gamma\big) \psi_1(K_a).
\end{split}
\end{equation}
From this and  Theorem \ref{th.2} one easily concludes the desired inequalities. 
\end{proof}

\begin{corollary}
\label{cor.2}
Let $p=1$, $\beta$ be a bounded stopping time, $\eta,\bar \eta\in \mathcal F_\beta$
and $f$ satisfy (H1)--(H5), (Z). Then there exists  $C>0$ - depending only on $\lambda,\mu,\esssup \beta, \gamma,\kappa$ -
such that 
\[
\mathbb E\big|\mathbb  E^{f}_{\alpha,\beta}(\eta)-\mathbb  E^{f}_{\alpha,\beta}(\bar \eta)\big|\le 
C(1+\|g\|_{L^1_{\mathbb F}})^2(1+\|\bar \eta\|_{L^1(\mathcal F_\beta)}+\|f(\cdot,0,0)\|_{L^1_{\mathbb F}(0,\beta)})^2\psi(\|\eta-\bar\eta\|_{L^1(\mathcal F_\beta)})
\]
for any stopping time $\alpha\le\beta$, where $\psi(x)=x+x^{\frac{1}{4} \kappa(1-\kappa^2)},\, x\ge 0$.
\end{corollary}
\begin{proof}
It follows directly from Corollary \ref{cor.1} applied with $q=\frac12(1+\kappa)$.
\end{proof}

\subsection*{Acknowledgements}
{\small T. Klimsiak is supported by Polish National Science Centre: Grant No. 2017/25/B/ST1/00878. M. Rzymowski acknowledges the support of the Polish National Science Centre: Grant No. 2018/31/N/ST1/00417.}

\end{document}